\documentclass[12pt, reqno]{amsart}
\usepackage{amsmath, amsthm, amscd, amsfonts, amssymb, graphicx, color}

\newtheorem{theorem}{Theorem}[section]
\newtheorem{lemma}[theorem]{Lemma}
\newtheorem{proposition}[theorem]{Proposition}
\newtheorem{corollary}[theorem]{Corollary}
\theoremstyle{definition}
\newtheorem{definition}[theorem]{Definition}

\theoremstyle{remark}
\newtheorem{remark}[theorem]{Remark}
\numberwithin{equation}{section}

\begin{document}
\setcounter{page}{1}

\title[Weighted EP Banach algebra elements]{Factorizations of weighted EP Banach space operators and Banach algebra elements}

\author[E. Boasso, D.S. Djordjevi\'c, D. Mosi\' c]{Enrico Boasso, Dragan S. Djordjevi\'c and Dijana Mosi\'c}



\begin{abstract}Weighted EP Banach space operators and Banach algebra elements are
characterized using different kinds of factorizations. The results
presented extend well-known characterizations of (weighted) EP
matrices, (weighted) EP Hilbert space operators and (weighted) EP
$C^*$-algebra elements.
\end{abstract} \maketitle

\section{Introduction and Preliminaries}

\noindent A square matrix $A$ is said to be EP if $N(A)=N(A^*)$.
Since a necessary and sufficient
condition for a matrix $A$ to be EP is the fact that
 $A$ commutes with its Moore-Penrose inverse, the notion under consideration
has been extended to Hilbert space operators and
$C^*$-algebra elements, see \cite{Pe, HM1, HM2, K1, Be}.\par

\indent In the context of Banach algebras the notion of Moore-Penrose inverse was
introduced by V. Rako\v cevi\' c in \cite{R1} (see also \cite{ B1}).
In the recent past EP Banach space operators and EP Banach algebra elements,
i.e., Moore-Penrose invertible operators or elements of
an algebra such that they commute with their Moore-Penrose
inverse, were introduced and characterized, see \cite{B1, B2R, B3}. It is worth noting that EP
objects generalize normal and hermitian objects, see \cite[Theorems 3.1
and 3.3]{B2R}.\par

\indent In addition, in the recent work by Y. Tian and H. Wang \cite{TW} the
notion of weighted EP matrices (matrices that commute with their
weighted Moore-Penrose inverse) was introduced. What is more, weighted EP Banach algebra elements
were characterized in \cite{BMDj}. \par

Note that one of the main lines of research concerning EP objects
consists in characterizing them through factorizations. In fact,
EP matrices, EP Hilbert and Banach space operators and EP
$C^*$-algebra and Banach algebra elements have been characterized
through several different kind of factorizations. \par

\indent The objective of this work is to characterize weighted EP
Banach space operators and weighted EP Banach algebra elements
using factorizations; actually, three different kinds of
factorizations will be considered (see sections 2-4). Several reasons motivate this
research. First of all, considering the general notion of weighted
EP Banach algebra element, results proved for (weighted) EP
matrices, Hilbert space operators and $C^*$-algebra elements can
be recovered and generalized. It is
also worth noting that due to the lack of involution on a Banach
algebra, and in particular on the Banach algebra of bounded and
linear maps defined on a Banach space, the proofs not only are
different from the ones known for matrices, Hilbert space
operators or $C^*$-algebra elements, but also they give a new
insight into the cases where the involution does exist. In fact,
the results and proofs presented do not depend on a particular
norm, the Euclidean norm, but they hold for any norm. In
particular, the results considered in this work also apply to
(weighted) EP matrices defined using an arbitrary norm on a finite
dimensional vector space, extending in this way the results of
\cite{TW} to any norm.\par

From now on $X$ and $Y$ will denote two Banach spaces and $L(X,Y)$
will stand for the Banach space of all bounded and linear maps defined on $X$ and
with values in $Y$. As usual, when $X=Y$, $L(X,Y)$ will be denoted by $L(X)$. If $T\in L(X,Y)$, then $N(T)\subseteq X$
and $R(T)\subseteq Y$ will stand for the null space and the range of $T$, respectively. In addition,
$I_X\in L(X)$ will denote the identity map. Moreover, $X^*$ will denote the dual space of $X$ and if
$T\in L(X)$, then $T^*\in L(X^*)$ will stand for the adjoint of $T\in L(X)$.\par

\indent On the other hand, $A$ will denote a complex unital Banach
algebra with unit $1$. In addition, the set of all invertible
elements of $A$ will be denoted by $A^{-1}$. If $a\in A$, then
$L_a \colon A\to A$ and $R_a\colon A\to A$ will denote the maps
defined by left and right multiplication, respectively: $
L_a(x)=ax$ and $R_a(x)=xa, $ where $x\in A$. Moreover, the
following notation will be used: $N(L_a)= a^{-1}(0)$,
$R(L_a)=aA$, $N(R_a)= a_{-1}(0)$,
$R(R_a)=Aa$ .\par

\indent Recall that an element $a\in A$ is called \it{regular},
\rm if it has a \it{generalized inverse}, \rm namely if there
exists $b\in A$ such that $ a=aba. $ Furthermore, a generalized
inverse $b$ of a regular element $a\in A$ will be called
\it{normalized}, \rm if $b$ is regular and $a$ is a generalized
inverse of $b$, equivalently, $a=aba$ and $b=bab$.
Note that if $b$ is a generalized inverse of $a$, then $c=bab$
is a normalized generalized inverse of $a$.\par

\indent Next follows the key notion in the definition of (weighted)
Moore-Penrose invertible Banach algebra elements.\par

\begin{definition}\label{def1}
Given a unital Banach algebra $A$, an element $a\in A$ is said to be \it hermitian, \rm
if $\parallel \exp(ita)\parallel =1$, for all $ t\in\Bbb R$.
\end{definition}

\indent Concerning equivalent definitions and the main properties of hermitian Banach
algebra elements and Banach space operators, see for example \cite{P,BD, D}.
Recall that if $A$ is a $C^*$-algebra, then
$a\in A$ is hermitian if and only if $a$ is self-adjoint,
see \cite[Proposition 20, Chapter I, Section 12]{BD}.
Given a unital Banach algebra $A$, the set of all Hermitian
elements of $A$ will be denoted by $H(A)$.
\par

\indent Now the notion of Moore-Penrose invertible Banach algebra element
will be recalled.\par

\begin{definition}\label{def2}
 Let $A$ be a unital Banach algebra and consider $a\in A$. If there exists
$x\in A$ such that $x$ is a normalized generalized inverse of $a$ and
$ax$ and $xa$ are hermitian, then $x$
will be said to be the Moore-Penrose inverse of $a$, and it will be
denoted by $a^{\dag}$.
\end{definition}

\indent Recall that according to \cite[Lemma 2.1]{R1}, there is at
most one Moore-Penrose inverse. Concerning the Moore-Penrose
inverse in Banach algebras, see \cite{R1,B1,B2R,B3}. For the
original definition of the Moore-Penrose inverse for matrices, see
\cite{Pe}.\par

\indent Next the definition of EP Banach algebra elements will be recalled.\par

\begin{definition} Given a unital Banach algebra $A$, the element $a\in A$ is said to be \it EP, \rm if $a^{\dag}$ exists and commutes with
$a$.
\end{definition}

\indent To recall the notion of weighted Moore-Penrose invertible Banach algebra elements,
some preparation is needed.\par

\indent Let $A$ be a complex unital Banach algebra and consider $a\in A$.
The element $a$ will be said to be \it positive\rm, if $V(a)\subset \mathbb R_+$,
where $V(a)=\{f(a)\colon f\in A^*, \parallel f\parallel \le 1, f(1)=1\}$ (\cite[Definition 5, Chapter V, Section 38]{BD}).
Denote $A_+$ the set of all positive elements of $A$. Note that necessary and sufficient for $a\in A$ to be positive
is that $a$ is hermitian and $\sigma (a)\subset \mathbb R_+$ (\cite[Definition 5, Chapter V, Section 38]{BD}).
Recall that according to \cite[Lemma 7, Chapter V, Section 38]{BD}, if $c\in A_+$,
then there exists $d\in A_+$ such that $d^2=c$. Moreover, according to \cite[Theorem]{G},
the square root of $c$ is unique. In particular, the square root of $c$ will be denoted
by $c^{1/2}$. For the definition and equivalent conditions of positive $C^*$-algebra elements,
see \cite[Definition 3.1 and Theorem 3.6, Chapter VIII, Section 3]{C}.
\par

\indent Given a complex unital Banach algebra $A$ and $u\in
A^{-1}\cap A_+$, denote by $A^u=(A^u, \parallel \cdot
\parallel_u)$ the complex unital Banach algebra with underlying
space $A$ and norm $\parallel x\parallel_u=\parallel
u^{1/2}xu^{-1/2}\parallel$. In the following definition weighted Moore-Penrose invertible
elements will be introduced.\par

\begin{definition}\label{def3}Let $A$ be a complex unital Banach algebra and consider
$e$ and $f$ two positive and invertible elements in $A$. The element $a\in A$
will be said to be \it weighted Moore-Penrose invertible with weights $e$ and $f$, \rm
if there exists $b\in A$ such that $b$ is a normalized generalized inverse of
$a$ and $ab$ (respectively $ba$) is a hermitian element of $(A^e, \parallel \cdot \parallel_e)$
(respectively $(A^f, \parallel \cdot \parallel_f)$.
\end{definition}

If the weighted Moore-Penrose inverse of $a$ exists, then it is
unique and so it will be denoted by $a_{e,f}^{\dag}$ (see \cite[Proposition 2.5]{BMDj}).
According to  \cite{BMDj}, if $A$ is a $C^*$-algebra, then the conditions in
Definition \ref{def3} are equivalent to the ones that characterize the usual
weighted Moore-Penrose inverse.\par

\indent Next the definition of weighted EP Banach algebra
element will be recalled.\par

\begin{definition}\label{def111} Given a unital Banach algebra $A$ and $e$, $f\in A$ two invertible
and positive elements, $a\in A$ is said to be \it weighted EP with
weights $e$ and $f$, \rm if $a^{\dag}_{e,f}$ exists and commutes with
$a$.
\end{definition}

\begin{remark}\rm Let $A$ be a unital Banach algebra and consider $a\in A$. \par
\noindent (i). Recall that $a\in A$ is said to be
 \it group invertible, \rm if there exists $b\in A$, a normalized generalized inverse of $a$, such that $ab=ba$.
It is well known that if the group inverse of $a$ exists, then it is unique; in addition, in this case it is denoted by $a^\sharp$.
Note that given $e$, $f\in A$ two invertible and positive elements,
necessary and sufficient for $a\in A$ to be weighted EP with weights $e$ and $f$ is that $a$ is group invertible and
$aa^\sharp$ is a hermitian element of $(A^e, \parallel \cdot \parallel_e)$ and $a^\sharp a$ is a hermitian element of $(A^f, \parallel \cdot \parallel_f)$.
Naturally, in this case, $a^\sharp=a^{\dag}_{e,f}$. So that, roughly speaking, weighted EP elements are group invertible elements with two
extra condition, i.e., the idempotents $aa^\sharp$ and $a^\sharp a$ must be hermitian elements of two associated Banach algebras.\par
\noindent (ii). Note that in Definition \ref{def111} the weights $e$, $f\in A$ need not to be ordered. In fact, according to a result
of hermitian elements (\cite[Theorem 4.4(i)]{D}), $a\in A$ is weighted EP with weights $e$ and $f$ if and only if
$a\in A$ is weighted EP with weights $f$ and $e$ (see Corollary \ref{cor5.9} and Theorem \ref{thm5.5}).
\end{remark}

\indent To study the factorization that will be considered in the
next section, the notion of weighted Moore-Penrose inverse for
operators defined between different Banach spaces needs to be
introduced. However, first some preliminary results will be
recalled.\par

\begin{remark}\label{rema21}\rm Let $X$ and $Y$ be two Banach spaces and let $T\in L(X,Y)$. If
$S\in L(Y,X)$ is a normalized generalized inverse of $T$, then it is not difficult to prove that
$TS\in L(Y)$ and $ST\in L(X)$ are idempotents, $R(TS)=R(T)$,
$N(TS)=N(S)$, $R(ST)=R(S)$, $N(ST)=N(T)$, $X=R(S)\oplus N(T)$ and
$Y=R(T)\oplus N(S)$.
\end{remark}

Next the notion of weighted Moore--Penrose inverse
for operators defined between different Banach spaces will be introduced.\par

\begin{definition}\label{def4} Let $X$ and $Y$ be two Banach spaces and let $T\in L(X,Y)$. Consider $E\in L(Y)$ and $F\in L(X)$
two invertible and positive operators. If there exists $S\in
L(Y,X)$ such that $TST=T$, $STS=S$, $TS\in H(L(Y)^E)$ and $ST\in
H(L(X)^F)$, then $T$ will be said to be \it weighted Moore-Penrose
invertible with weights $E$ and $F$.\rm
\end{definition}

\indent In first place it will be shown that given $T\in L(X,Y)$, there is at most one
operator $S\in L(Y,X)$ satisfying Definition \ref{def4}. In the following Lemma, ideas similar to the ones
in \cite[Lemma 2.1]{R1} will be used. In addition, it will be used the following result. Given $Z$ a Banach space and
$E$, $F\in L(Z)$ two hermitian idempotents such that $R(E)=R(F)$, then $E=F$ (\cite[Theorem 2.2]{P}). \par

\begin{lemma} \label{lem5}Let $X$ and $Y$ be two Banach spaces and let $T\in L(X,Y)$. Consider $E\in L(Y)$ and $F\in L(X)$
two invertible and positive operators. If $S_i\in L(Y,X)$ complies
the four conditions of Definition \ref{def4}, $i=1,2$, then
$S_1=S_2$.
\end{lemma}
\begin{proof} Since $R(TS_1)=R(T)=R(TS_2)$ and $TS_i\in H(L(Y)^E)$, $i=1, 2$, according to \cite[Theorem 2.2]{P},
$TS_1=TS_2$. In addition, since $R(I_X-S_1T)=N(S_1T)=N(T)=N(S_2T)=R(I_X-S_2T)$ and $ST_i\in
H(L(X)^F)$, $i=1, 2$,
according to \cite[Theorem 44(i)]{D} and \cite[Theorem 2.2]{P}, $I_X-S_1T=I_X-S_2T$, equivalently $S_1T=S_2T$.
Thus,
$$
S_1=S_1TS_1=S_1TS_2=S_2TS_2=S_2.
$$
\end{proof}

\indent The following result will be used in the next section.\par

\begin{lemma}\label{lem10} Let $X$ and $Y$ be two Banach spaces and let $T\in L(X,Y)$. Consider $E\in L(Y)$ and $F\in L(X)$
two invertible and positive operators. Necessary and sufficient
for $T^{\dag}_{E,F}$ to exit is that there exist two idempotents
$P\in H(L(Y)^E)$ and $Q\in H(L(X)^F)$ such that $R(P)=R(T)$ and
$N(Q)=N(T)$.
\end{lemma}

\begin{proof} Adapt the proof of \cite[Theorem 2.7]{BMDj}
to the conditions of this lemma.
\end{proof}
\section{Factorization $a=bc$}

In this section, given a unital Banach algebra $A$ and $e$, $f\in A$ invertible and positive, weighted EP elements of the form $a=bc$ will be
characterized, where $a$, $b$, $c\in A$, $a$ is weighted Moore-Penrose invertible with weights $e$ and $f$,
 $b^{-1}(0)=\{0\}$ and $cA=A$. However, in first place the Banach space operator case will be studied.\par

\begin{theorem}\label{thm1f} Let $X$ and $Y$ be two Banach spaces and consider $E,H\in L(X)$ and $F\in L(Y)$ three invertible
positive operators. Let $T\in L(X)$ such that $T^\dag_{E,H}$
exists and suppose that there exist $C\in L(X,Y)$ and $B\in L(Y,X)$ such that $C$ is
surjective, $B$ is injective and $T=BC$. Then the following
statements hold.
\begin{itemize}

\item[\rm (i)] There exists $B^\dag_{E,F}\in L(X,Y)$ such that
$B^\dag_{E,F}B=I_Y$.

\item[\rm(ii)] There exists $C^\dag_{F,H}\in L(Y,X)$ such that
$CC^\dag_{F,H}=I_Y$.

\item[\rm(iii)] $T^\dag_{E,H}=C^\dag_{F,H}B^\dag_{E,F}$,
$TT^\dag_{E,H}=BB^\dag_{E,F}$, $T^\dag_{E,H}T=C^\dag_{F,H}C$,
$B^\dag_{E,F}=CT^\dag_{E,H}$, $C^\dag_{F,H}=T^\dag_{E,H}B$,
$TC^\dag_{F,H}=B$ and $B^\dag_{E,F}T=C$.
\end{itemize}
\end{theorem}

\begin{proof} Note that according to \cite[Theorem 2.7]{BMDj}, there exist two idempotents
$P$,$Q\in L(X)$ such that $P\in H(L(X)^E)$, $Q\in H(L(X)^H)$,
$R(P)=R(T)$ and $N(Q)=N(T)$.\par

\indent (i). Since $B$ is injective and $R(P)=R(T)=R(B)$,
according to Lemma \ref{lem10}, $B^\dag_{E,F}$ exists.
In fact, $I_Y\in H(L(Y)^H)$ and $P\in H(L(X)^E)$ are two idempotents such that
$N(B)=N(I_Y)$ and $R(B)=R(P)$.
In addition, $R(I_Y-B^\dag_{E,F}B)=N(B^\dag_{E,F}B)=N(B)=0$.\par

\indent (ii). Similarly, since $I_Y\in H(L(Y)^F)$ and $Q\in H(L(X)^H)$ are two
idempotents such that $R(C)=Y=R(I_Y)$ and $N(C)=N(T)=N(Q)$, according to
Lemma \ref{lem10}, $C^\dag_{F,H}$ exists.
Further, since $CC^\dag_{F,H}$ and $I_Y\in H(L(Y)^F)$ are idempotents such
that $R(CC^\dag_{F,H})=R(C)=R(I_Y)$, according to \cite[Theorem 2.2]{P},
$CC^\dag_{F,H}=I_Y$.\par

\indent (iii). Consider $S=C^\dag_{F,H}B^\dag_{E,F}$. It is not
difficult to prove that $T=TST$, $S=STS$, $TS=BB^\dag_{E,F}$,
$ST=C^\dag_{F,H}C$. However, since $BB^\dag_{E,F}\in H(L(X)^E)$
and $C^\dag_{F,H}C\in H(L(X)^H)$ are idempotent operators,
according to Definition \ref{def4} and Lemma \ref{lem5},
$T^\dag_{E,H}=C^\dag_{F,H}B^\dag_{E,F}$. The remaining identities
can be derived from what has been proved and statements (i)-(ii).
\end{proof}

\begin{remark}\rm (a). Let $X$ be a Banach space and consider $E,H\in L(X)$ two invertible
positive operators. Let $T\in L(X)$ such that $T^\dag_{E,H}$. Note that the decomposition of $T$ as in
Theorem \ref{thm1f} is always possible. In fact, since $R(T)$ is closed, $T=\overline{T}\pi$,
where $\pi\colon X\to X/N(T)=Y$ is the canonical quotient map and $\overline{T}\colon Y\to X$
is the factorization of $T$. In addition, consider any invertible and positive operator $F\in L(Y)$,
for example $F=I_Y$.\par
\noindent (b). Note that in Theorem \ref{thm1f} the identity $T^\dag_{E,H}=C^\dag_{F,H}B^\dag_{E,F}$
is a particular reverse order law for $T=BC$.
\end{remark}
\indent In the following theorem weighted EP operators of the form $T=BC$ will
be characterized.\par

\begin{theorem}\label{thm2f} Under the same hypotheses of Theorem \ref{thm1f}, the following statements are
equivalent.
\begin{itemize}
\item[\rm (i)] $T$ is weighted EP with weights $E$ and $H$.

\item[\rm(ii)] $BB^\dag_{E,F}=C^\dag_{F,H}C$.

\item[\rm(iii)] $R(B)=R(C^\dag_{F,H})$ and $N(B^\dag_{E,F})=N(C)$.

\item[\rm(iv)] $(I_X-C^\dag_{F,H}C)B=0$ and
$C(I_X-BB^\dag_{E,F})=0$.

\item[\rm(v)] $B^\dag_{E,F}(I_X-C^\dag_{F,H}C)=0$ and
$(I_X-BB^\dag_{E,F})C^\dag_{F,H}=0$.

\item[\rm(vi)] There exists an isomorphism $U\in L(Y)$ such that
$C=UB^\dag_{E,F}$ and $B=C^\dag_{F,H}U$.

\item[\rm(vii)] There exist a surjective map $U_1\in L(Y)$ and an
injective map $U_2\in L(Y)$ such that $C=U_2B^\dag_{E,F}$ and
$B=C^\dag_{F,H}U_1$.

\item[\rm(viii)] There exist $U_3,U_4,U_5,U_6\in L(Y)$ such that
$C=U_5B^\dag_{E,F}$, $B^\dag_{E,F}=U_6C$, $B=C^\dag_{F,H}U_3$ and
$C^\dag_{F,H}=BU_4$.

\item[\rm(ix)] There exist a surjective map $U_7\in L(Y)$ and an
injective map $U_8\in L(Y)$ such that $B^\dag_{E,F}=U_8C$ and
$C^\dag_{F,H}=BU_7$.
\end{itemize}
\end{theorem}

\begin{proof} According to Theorem \ref{thm1f}, statements (i) and (ii) are equivalents.
In addition, according to Remark \ref{rema21}, statement (ii) implies statement (iii).
On the other hand, if statement (iii) holds, then
$R(BB^\dag_{E,F})=R(B)=R(C^\dag_{F,H})=R(C^\dag_{F,H}C)$
and $N(BB^\dag_{E,F})=N(B^\dag_{E,F})=N(C)=N(C^\dag_{F,H}C)$.
However, since $BB^\dag_{E,F}$ and $C^\dag_{F,H}C$ are idempotents,
statement (ii) holds.\par

Suppose that the statement (iii) holds. Since
$N(I_X-C^\dag_{F,H}C)=R(C^\dag_{F,H}C)=R(C^\dag_{F,H})=R(B)$ and
$R(I_X-BB^\dag_{E,F})=N(BB^\dag_{E,F})=N(B^\dag_{E,F})=N(C)$,
statement (iv) holds.

On the other hand, observe that if statement (iv) holds, then $R(T)=R(B)\subseteq
N(I_X-C^\dag_{F,H}C)=R(C^\dag_{F,H})=R(T^{\dag}_{E,H})$ and
$N(T^{\dag}_{E,H})=N(B^\dag_{E,F})=R(I_X-BB^\dag_{E,F})\subseteq
N(C)=N(T)$. However, since according to Remark \ref{rema21}
$$
X=R(T)\oplus N(T^{\dag}_{E,H})=R(T^{\dag}_{E,H})\oplus N(T),
$$
it is not difficult to prove that $R(T)=R(T^{\dag}_{E,H})$ and
$N(T^{\dag}_{E,H})=N(T)$. As a result,
$R(TT^{\dag}_{E,H})=R(T^{\dag}_{E,H}T)$ and
$N(TT^{\dag}_{E,H})=N(T^{\dag}_{E,H}T)$. Therefore,
$TT^{\dag}_{E,H}=T^{\dag}_{E,H}T$.

A similar argument proves that statement (iii) implies statement (v), which in turn implies
statement (i).\par

Next suppose that statement (i) holds and consider $U=CB\in L(Y)$. Then, according to statement (iv),
$C=UB^\dag_{E,F}$ and $B=C^\dag_{F,H}U$. In order
to prove that $U\in L(Y)$ is an isomorphism, define
$Z=B^\dag_{E,F}C^\dag_{F,H}\in L(Y)$. Now well, according to statement (iv)
and to Theorem \ref{thm1f}(ii), $UZ=I_Y$. In addition, according again to
statement (iv) and Theorem \ref{thm1f}(i), $ZU=I_Y$. \par

Clearly, statement (vi) implies statements (vii)-(ix). On the
other hand, it is not difficult to prove that statement (vii)
(respectively statements (viii) and (ix)) implies statement (iii).
\end{proof}

\indent Next the Banach algebra case will be studied. Firstly some
preliminary facts need to be considered.\par

\begin{theorem}\label{thm3f} Let $A$ be a unital Banach algebra and consider three invertible and positive elements $e,f,h\in A$
and $a\in A$ such that $a^\dag_{e,h}$ exists. Suppose that there exist $b,c\in
A$ such that $(b)^{-1}(0)=\{0\}$, $cA=A$ and $a=bc$. Then,
the following statements hold:
\begin{itemize}

\item[\rm (i)] There exists $b^\dag_{e,f}\in A$ such that
$b^\dag_{e,f}b=1$.

\item[\rm(ii)] There exists $c^\dag_{f,h}\in A$ such that
$cc^\dag_{f,h}=1$.

\item[\rm(iii)] $a^\dag_{e,h}=c^\dag_{f,h}b^\dag_{e,f}$,
$aa^\dag_{e,h}=bb^\dag_{e,f}$, $a^\dag_{e,h}a=c^\dag_{f,h}c$,
$b^\dag_{e,f}=ca^\dag_{e,h}$, $c^\dag_{f,h}=a^\dag_{e,h}b$, $ac^\dag_{f,h}=b$
and $b^\dag_{e,f}a=c$.
\end{itemize}
\end{theorem}

\begin{proof} Consider the maps $L_a, L_b, L_c\in L(A)$. Note that
$L_a=L_bL_c$, $N(L_b)=\{0\}$ and $R(L_c)=A$. In addition, according to
\cite[Theorem 2.4]{BMDj}, $L_e$, $L_f$ and $L_h\in L(A)$ are three positive
and invertible operators. Now well, according to \cite[Theorem
2.8]{BMDj}, $(L_a)_{L_e,L_h}^{\dag}$ exists and $(L_a)_{L_e,L_h}^{\dag}=L_{a_{e,h}^{\dag}}\in L(A)$.
Consequently, according to Theorem \ref{thm1f}, $(L_b)_{L_e,L_f}^{\dag}$ exists
and $(L_b)_{L_e,L_f}^{\dag}L_b=I_A$, $(L_c)_{L_f,L_h}^{\dag}$
exists and $L_c(L_c)_{L_f,L_h}^{\dag}=I_A$,
$L_{a_{e,h}^{\dag}}=(L_c)_{L_f,L_h}^{\dag}(L_b)_{L_e,L_f}^{\dag}$,
$(L_b)_{L_e,L_f}^{\dag}=L_{ca^\dag_{e,h}}$ and
$(L_c)_{L_f,L_h}^{\dag}=L_{a^\dag_{e,h}b}$.

Let $c'=(L_c)_{L_f,L_h}^{\dag}(1)=a^\dag_{e,h}b$. Since
$L_c(L_c)_{L_f,L_h}^{\dag}=I_A$, $cc'=1$. In particular, $cc'c=c$,
$c'cc'=c'$ and $c'c=a^\dag_{e,h}bc=a^\dag_{e,h}a\in H(A^h)$.
Hence, $c^\dag_{f,h}$ exists, $c^\dag_{f,h}=c'=a^\dag_{e,h}b$ and
$cc^\dag_{f,h}=1$.

If $b'=ca^\dag_{e,h}$, then
$b'b=c(a^\dag_{e,h}b)=cc^\dag_{f,h}=1$. As a result, $bb'b=b$ and
$b'bb'=b'$. However, since $bb'=bca^\dag_{e,h}=aa^\dag_{e,h}\in H(A^e)$,
$b^\dag_{e,f}$ exists,
$b^\dag_{e,f}=b'=ca^\dag_{e,h}$ and $b^\dag_{e,f}b=1$.

Set $a'=c^\dag_{f,h}b^\dag_{e,f}$. According to what has been proved, $aa'a=a$, $a'aa'=a'$,
$aa'=bb^\dag_{e,f}$ and $a'a=c^\dag_{f,h}c$. Therefore, according to \cite[Theorem 2.5]{BMDj}
$a^\dag_{e,h}=a'=c^\dag_{f,h}b^\dag_{e,f}$. The remaining identities can be derived from statements
(i)-(ii).
\end{proof}

\indent In the following theorem, weighted Moore-Penrose invertible Banach algebra elements
of the form $a=bc$ will characterized.\par

\begin{theorem}\label{thm4f} Under the same hypotheses of Theorem
\ref{thm3f}, the following statements are
equivalent.
\begin{itemize}
\item[\rm (i)] The element $a$ is weighted EP with weights $e$ and $h$.

\item[\rm(ii)] $bb^\dag_{e,f}=c^\dag_{f,h}c$.

\item[\rm(iii)] $bA=c^\dag_{f,h}A$ and
$(b^\dag_{e,f})^{-1}(0)=c^{-1}(0)$.

\item[\rm(iv)] $(1-c^\dag_{f,h}c)b=0$ and $c(1-bb^\dag_{e,f})=0$.

\item[\rm(v)] $b^\dag_{e,f}(1-c^\dag_{f,h}c)=0$ and
$(1-bb^\dag_{e,f})c^\dag_{f,h}=0$.

\item[\rm(vi)] There exists $u\in A^{-1}$ such that $c=ub^\dag_{e,f}$ and
$b=c^\dag_{f,h}u$.

\item[\rm(vii)] There exist $u_1,u_2\in A$ such that $u_1A=A$, $u_2^{-1}(0)=\{0\}$,
$c=u_2b^\dag_{e,f}$ and $b=c^\dag_{f,h}u_1$.

\item[\rm(viii)] There exist $u_3, u_4, u_5,u_6\in A$ such that $c=u_5b^\dag_{e,f}$,
$b^\dag_{e,f}=u_6c$, $b=c^\dag_{f,h}u_3$ and $c^\dag_{f,h}=bu_4$.

\item[\rm(ix)] There exist $u_7,u_8\in A$ such that $u_7A=A$,
$u_8^{-1}(0)=\{0\}$, $b^\dag_{e,f}=u_8c$ and $c^\dag_{f,h}=bu_7$.

\item[\rm(x)] $a\in c^{\dag}_{f,h}A\cap Ab^{\dag}_{e,f}$.

\item[\rm(xi)] $a^{\dag}_{e,h}\in bA\cap Ac$.

\item[\rm(xii)] $bA^{-1}=c^{\dag}_{f,h}A^{-1}$ and $A^{-1}c= A^{-1} b^\dag_{e,f}$.

\item[\rm(xiii)] $b_{-1}(0)=(c^{\dag}_{f,h})_{-1}(0)$ and $Ac=Ab^\dag_{e,f}$.
\end{itemize}
\end{theorem}

\begin{proof} As in the proof of Theorem \ref{thm3f}, let $T=L_a$, $B=L_b$, $C=L_c\in L(A)$ and note that
$T=BC$, $N(B)=\{0\}$ and $R(C)=A$. In addition,
$E=L_e$, $F=L_f$, $H=L_h\in L(A)$ are three
invertible positive operators (\cite[Lemma 2.4]{BMDj}). Therefore, according to \cite[Theorem 2.8]{BMDj} and Theorem \ref{thm1f}, $T,B,C$ are weighted
Moore--Penrose invertible and $T^\dag_{E,H}=L_{a_{e,h}^{\dag}}$,
$B^\dag_{E,F}=L_{b_{e,f}^{\dag}}$ and
$C^\dag_{F,H}=L_{c_{f,h}^{\dag}}$. Now well, according to Theorem \ref{thm2f}, statements (i)-(v) are
equivalent. To prove the equivalence of statements (i) and (vi)-(ix), adapt the
proof of Theorem \ref{thm2f} to the case under consideration.\par

\indent Suppose that statement (i) holds. Note that a straightforward calculation, using in particular
statement (vi), proves that
that $a\in c^{\dag}_{f,h}A\cap Ab^{\dag}_{e,f}$.
On the other hand, if statement (x) holds, then, according to Theorem \ref{thm3f},
$a\in a^{\dag}_{e,h}A\cap Aa^{\dag}_{e,h}$. In particular,
$aA\subseteq a^{\dag}_{e,h}A$ and $(a^{\dag}_{e,h})^{-1}(0)\subseteq a^{-1}(0)$.
However, an argument similar to the one in Theorem \ref{thm2f} (statement (iv)
implies statement (i)), using in particular \cite[Remark 2.6]{BMDj}, proves that
statement (i) holds. \par

\indent Note that according to \cite[Remark 2.6(i)]{BMDj},
$(a^\dag_{e,h})^{\dag}_{h,e}=a$. In addition, according to Theorem
\ref{thm3f}, $(b_{e,f}^{\dag})^{\dag}_{f,e}=b$ and
$(c_{f,h}^{\dag})^{\dag}_{h,f}=c$. What is more, it is clear that
$a$ is weighted EP with weights $e$ and $h$ if and only if
$a^\dag_{e,h}$ is weighted EP with weights $h$ and $e$. Therefore,
to prove the equivalence between statements (i) and (xi), it is
enough to apply the equivalence between statements (i) a (x) to
$a^\dag_{e,h}=c^\dag_{f,h}b^\dag_{e,f}$.\par

\indent Clearly, statement (vi) and (xii) are equivalent.\par

\indent Observe that statement (i) is equivalent to $N(R_{aa^\dag_{e,h}})=N(R_{a^\dag_{e,h}a})$ and
$R(R_{aa^\dag_{e,h}})=R(R_{a^\dag_{e,h}a})$. In addition, it is not difficult to prove that
$R(R_{aa^\dag_{e,h}})= R(R_{bb^\dag_{e,f}})=Ab^\dag_{e,f}$ and
$R(R_{a^\dag_{e,h}a})=R(R_{c_{f,h}^{\dag}c})=Ac$; similarly, $(aa^\dag_{e,h})_{-1}(0)=(bb^{\dag}_{e,f})_{-1}(0)=b_{-1}(0)$
and $(a^\dag_{e,h}a)_{-1}(0)= (c_{f,h}^{\dag}c)_{-1}(0)=(c_{f,h}^{\dag})_{-1}(0)$ (use the fact that $a^\dag_{e,h}$
(respectively $b^\dag_{e,f}$, $c_{f,h}^{\dag}$) is a normalized generalized inverse of $a$ (respectively $b$, $c$).
Therefore, statement (i) and (xiii) are equivalent.
\end{proof}

\section{Factorization $a^\dag_{e,f}=sa$}

\indent In the following theorem, given a unital Banach algebra
$A$, elements of the form $a^\dag_{e,f}=sa$ will be characterized
($a$, $s\in A$, $e$, $f\in A$ invertible and positive, $a$
weighted Moore-Penrose invertible with weights $e$ and $f$).
Recall that according to \cite[Remark 2.6]{BMDj},
$a_{e,f}^{\dag}A=a_{e,f}^{\dag}aA$, $aa_{e,f}^{\dag}A=aA$,
$(aa_{e,f}^{\dag})^{-1}(0) =(a_{e,f}^{\dag})^{-1}(0)$, $(a_{e,f}^{\dag}a)^{-1}(0) =a^{-1}(0)$ and
$A=a_{e,f}^{\dag}A\oplus a^{-1}(0)= aA\oplus (a_{e,f}^{\dag})^{-1}(0)$.\par

\begin{theorem}
Let $A$ be a unital Banach algebra and consider two invertible positive elements $e,f\in A$.
Let $a\in A$ such that $a^\dag_{e,f}$ exists. Then the
following statements are equivalent.
\begin{itemize}

\item[\rm(i)] The element $a$ is weighted EP with weights $e$ and $f$.

\item[\rm(ii)] There exist $s,t\in A$ such that $s^{-1}(0)=\{0\}$, $tA=A$ and
$a^\dag_{e,f}=sa=at$.

\item[\rm(iii)] There exist $s_1,t_1\in A$ such that $a^\dag_{e,f}=s_1a=at_1$.

\item[\rm(iv)] There exist $u,v,u_1,v_1\in A$ such that
$a^\dag_{e,f}a=ua^\dag_{e,f}=av$ and
$aa^\dag_{e,f}=a^\dag_{e,f}u_1=v_1a$.

\item[\rm(v)] There exist $u_2,v_2,u_3,v_3\in A$ such that
$a^\dag_{e,f}a=u_2a^\dag_{e,f}$, $aa^\dag_{e,f}=a^\dag_{e,f}u_3$ and
$a^\dag_{e,f}=av_2=v_3a$.

\item[\rm(vi)] There exist $x,y\in A^{-1}$ such that
$a^\dag_{e,f}a=xaa^\dag_{e,f}=aa^\dag_{e,f}y$.

\item[\rm(vii)] There exist $x_1,y_1\in A$ such that $x_1^{-1}(0)=\{0\}$,
$y_1A=A$ and $a^\dag_{e,f}a=x_1aa^\dag_{e,f}=aa^\dag_{e,f}y_1$.

\item[\rm(viii)] There exist $x_2,y_2\in A$
such that $x_2^{-1}(0)=\{0\}$ and $a^\dag_{e,f}a=x_2aa^\dag_{e,f}=aa^\dag_{e,f}y_2$.

\item[\rm(ix)] There exist $z_1$, $z_2\in A$ such that $a^\dag_{e,f}a=az_1a^\dag_{e,f}$ and
$aa^\dag_{e,f}=a^\dag_{e,f}z_2a$.
\end{itemize}
\end{theorem}

\begin{proof} If statement (i) holds, according to \cite[Theorem 3.7(xiv)]{BMDj},
$a\in a^\dagger_{e,f} A^{-1}\cap A^{-1}
a^\dagger_{e,f}$. In particular, there
exist $s,t\in A^{-1}$ such that $a^\dag_{e,f}=sa=at$, which implies
statement (ii). \par

\indent Note that statement (ii) clearly implies statement (iii).
On the other hand, if statement (iii) holds, then
$a^\dag_{e,f}A\subset aA$ and $a^{-1}(0)\subset
(a^\dag_{e,f})^{-1}(0)$. However, since $A=a_{e,f}^{\dag}A\oplus a^{-1}(0)= aA\oplus (a_{e,f}^{\dag})^{-1}(0)$
(\cite[Remark 2.6(iv)]{BMDj}), it is not difficult to prove that
$a^\dag_{e,f}A= aA$ and $a^{-1}(0)=(a^\dag_{e,f})^{-1}(0)$.
Therefore, according to \cite[Theorem 3.7(ii)]{BMDj}, statement (i) holds.\par

\indent Clearly, statement (i) implies statement (iv).
On the other hand, if statement (iv) holds, a straightforward
calculation, using in particular \cite[Remark 2.6]{BMDj}, proves that
$aA=a^\dag_{e,f}A$ and $a^{-1}(0)=(a^\dag_{e,f})^{-1}(0)$.
In particular, $R(L_{aa^\dag_{e,f}})=R(L_{a^\dag_{e,f}a})$ and
$N(L_{aa^\dag_{e,f}})=N(L_{a^\dag_{e,f}a})$. Since, $aa^\dag_{e,f}$ and
$a^\dag_{e,f}a$ are idempotents,
$a$ is weighted EP with weights $e$ and $h$.\par

A similar argument proves the equivalence between statements (i) and (v).\par

Clearly, (i) $\Rightarrow$ (vi) $\Rightarrow$ (vii) $\Rightarrow$
(viii). Suppose that (viii) holds. Then, according to \cite[Remark
2.6]{BMDj}, $a^\dag_{e,f}A\subseteq aA$ and $a^{-1}(0)\subseteq
(a^\dag_{e,f})^{-1}(0)$. However,  using an argument similar to
the one in the proof of Theorem \ref{thm4f} (statement (x) implies
statement (i)), it is easy to prove that statement (i) holds.\par

It is clear that statement (i) implies statement (ix). On the
other hand, if statement (ix) holds, then $a^\dag_{e,f}A=aA$ and
$(a^\dag_{e,f})^{-1}(0)=a^{-1}(0)$. To conclude the proof, proceed
as before (statement (iv) implies statement (i)).
\end{proof}

\section{ Factorization of the form $a=ucu^{-1}$}

\noindent Let $H$ be a Hilbert space and consider $T\in L(H)$. It is well known that necessary and sufficient for $T$ to be EP
is that there exists an orthogonal idempotent $P\in L(H)$ such that if $H_1=R(P)$, $H_2=N(P)$, $H=H_1\oplus H_2$ and
$T_1=T\mid_{H_1}^{H_1}\colon H_1\to H_1$, then $T_1\in L(H_1)$ is an isomorphism and
with respect to the aforementioned decomposition, $T$ can be presented as
$$
T= \begin{pmatrix} T_1&0\\
             0&0\\
\end{pmatrix},
$$
see for example \cite[Theorem 2.1]{DKS}. The case of EP Banach space operators was considered in
\cite[Section 5]{B3}. Next this kind of factorization will be studied for
weighted EP bounded and linear maps. However, in first place some preparation is needed.\par

\indent In the following proposition, given two Banach spaces $X_1$ and $X_2$, $X_1\oplus_1X_2$ will denote the Banach space
$X_1\oplus X_2$ with the $1$-norm, i.e., $\parallel x_1\oplus x_2\parallel=\parallel x_1\parallel + \parallel x_2\parallel$.
\par

\begin{proposition}\label{prop5.2}Let $X_1$ and $X_2$ be two Banach spaces and consider $T_1\in L(X_1)$
an isomorphic operator. Let $X$ be a Banach space and consider $E$, $F$, $T\in L(X)$ such that $E$ and $F$ are two invertible
positive operators and there exists
a linear and bounded isomorphism $J\colon X_1\oplus_1 X_2\to X$ with the property
$T=J(T_1\oplus 0)J^{-1}$. Consider $T'=J(T_1^{-1}\oplus 0)J^{-1}\in L(X)$. Then, the following statements are equivalent.\par
\noindent \rm (i) \it $T$ is weighted Moore-Penrose invertible with weights $E$ and $F$ and $T^\dag_{E,F}=T'$.\par
\noindent \rm (ii) \it $T$ is weighted EP with weights $E$ and $F$ and $T^\dag_{E,F}=T'$.\par
\noindent \rm (iii) \it $Q_1=J(I_{X_1}\oplus 0)J^{-1}\in H(L(X)^E)$ and $Q_2=J(0\oplus I_{X_2})J^{-1}\in H(L(X)^F)$.
 \end{proposition}
\begin{proof} It is not difficult to prove that $T'$ is a normalized generalized
inverse of $T$ and that $Q_1$ and $Q_2$ are the projections onto the closed and complemented
subspaces $J(X_1\oplus 0)$ and $J(0\oplus X_2)$ respectively. Furthermore, according to \cite[Theorem 4.4(i)]{D}, since $TT'=T'T=Q_1=I-Q_2$,
statements (i)-(iii) are equivalent.
\end{proof}

\indent In the following theorem, weighted $EP$ bounded and linear maps will be characterized.\par

\begin{theorem}\label{thm5.3} Let $X$ be a Banach space and consider $E$, $F\in L(X)$ two invertible positive operators.
Let $T\in L(X)$. Then, the following
statements are equivalent.\par
\noindent \rm (i) \it $T$ is weighted EP with weights $E$ and $F$,\par
\noindent\rm (ii) \it There exist two Banach spaces $X_1$ and $X_2$, $T_1\in L(X_1)$ an isomorphic
operator, and $J\colon X_1\oplus_1 X_2\to X$ a linear and bounded isomorphism such that
$T=J(T_1\oplus 0)J^{-1}$, $T^\dag_{E,F}=J(T_1^{-1}\oplus 0)J^{-1}$, $J(I_{X_1}\oplus 0)J^{-1}\in H(L(X)^E)$ and $J(0\oplus I_{X_2})J^{-1}\in H(L(X)^F)$.\par
\end{theorem}
\begin{proof} According to Proposition \ref{prop5.2}, statement (ii) implies that $T$ is weighted EP
with weights $E$ and $F$.
On the other hand, if statement (i) holds, according to \cite[Theorem 3.4(a)]{BMDj},
there exist $P\in H(L(X)^E)\cap H(L(X)^F)$ such that $N(P)=N(T)$ and $R(P)=R(T)$. Denote then
$X_1=R(P)$, $X_2=N(P)$ and $T_1=T\mid_{X_1}^{X_1}\colon X_1\to X_1$. It is clear that $T_1\in L(X_1)$
is an isomorphism. Moreover, $J\colon X_1\oplus_1 X_2\to X$ is the map
$J(x_1\oplus x_2)= x_1+x_2$. Since $J^{-1}\colon X\to X_1\oplus_1 X_2$ is such that
$J^{-1}=P\oplus (I_X-P)$, $T=J(T_1\oplus 0)J^{-1}$, $T^\dag_{E,F}=J(T_1^{-1}\oplus 0)J^{-1}$, $ J(I_{X_1}\oplus 0)J^{-1}=P\in H(L(X)^E)$
and $ J(0\oplus I_{X_2})J^{-1}=I-P\in H(L(X)^F)$ (\cite[Theorem 4.4(i)]{D}).
\end{proof}

\indent Next weighted EP Banach space operators will be characterized in terms of injective and surjective bounded and linear
maps.\par

\begin{theorem}\label{thm5.4}Let $X$ be a Banach space and consider $E$ and $F$ two invertible positive operators.
Let $T\in L(X)$. Then, the
following statements are equivalent.\par
\noindent \rm (i) \it $T$ is weighted EP with weights $E$ and $F$.\par
\noindent \rm (ii) \it There exist Banach spaces $X_1$ and $X_2$, $T_1\in L(X_1)$ an isomorphism, $S\in L(X_1\oplus_1 X_2, X)$
injective, $U\in L(X,X_1\oplus_1 X_2)$ surjective, and an idempotent $P\in H(L(X)^E)\cap H(L(X)^F)$ such that
$T=S(T_1\oplus 0)U$, $R(P)=S(X_1\oplus 0)$ and $N(P)=U^{-1}(0\oplus X_2)$.
\end{theorem}
\begin{proof} If $T$ is weighted EP with weights $E$ and $F$, then let $X_1$, $X_2$ and $T_1\in L(X_1)$ be as in Theorem \ref{thm5.3} and define
$S=J\in L(X_1\oplus_1 X_2, X)$ and $U=J^{-1}\in L(X,X_1\oplus_1
X_2)$, $J$ as in Theorem \ref{thm5.3}. Moreover, define
$P=J(I_{X_1}\oplus 0)J^{-1}\in H(L(X)^E)\cap H(L(X)^F)$. Since
$R(P)=R(T)=S(X_1\oplus 0)$ and $N(P)=N(T)=U^{-1}(0\oplus X_2)$,
statement (ii) holds.\par

\indent On the other hand, if statement (ii) holds,
then $P$ is an idempotent such that $P\in H(L(X)^E)\cap H(L(X)^F)$, $R(P)=R(T)$ and $N(P)=N(T)$.
Therefore, according to \cite[Theorem 3.4(a)]{BMDj}, $T$ is weighted EP with weights $E$ and $F$.
\end{proof}

\indent As a corollary of Theorem \ref{thm5.3}, more characterizations of weighted EP Banach space operators
will be given.

\begin{corollary}\label{cor5.9}Let $X$ be a Banach space and consider $E$ and $F$ two invertible positive operators.
Let $T\in L(X)$. Then, the following statements are equivalent.\par
\noindent \rm (i) \it $T$ is weighted EP with weights $E$ and $F$.\par
\noindent \rm (ii) \it $T$ is weighted EP with weights $F$ and $E$.\par
\noindent \rm (iii) \it $T$ is weighted EP both with weights $E$ and $E$ and with weights $F$ and $F$.\par
\end{corollary}
\begin{proof} According to Theorem \ref{thm5.3}, statement (ii) is equivalent to the fact that
there exist two Banach spaces $X_1$ and $X_2$, $T_1\in L(X_1)$ an isomorphic
operator, and $J\colon X_1\oplus_1 X_2\to X$ a linear and bounded isomorphism such that
$T=J(T_1\oplus 0)J^{-1}$, $T^\dag_{E,F}=J(T_1^{-1}\oplus 0)J^{-1}$, $J(I_{X_1}\oplus 0)J^{-1}\in H(L(X)^E)$ and $J(0\oplus I_{X_2})J^{-1}\in H(L(X)^F)$.
Now well, according to \cite[Theorem 4.4(i)]{D}, $J(I_{X_1}\oplus 0)J^{-1}\in H(L(X)^E)$ if and only if
 $J(0\oplus I_{X_2})J^{-1}=I-J(I_{X_1}\oplus 0)J^{-1}\in H(L(X)^E)$. Similarly,
$J(0\oplus I_{X_2})J^{-1}\in H(L(X)^F)$ if and only if
$J(I_{X_1}\oplus 0)J^{-1}=I-J(0\oplus I_{X_2})J^{-1}\in H(L(X)^F)$. As a result, statement (i) and
(ii) are equivalent.\par
\indent Clearly, statements (i) and (ii) imply statement (iii), and
statement (iii) implies both statements (i) and (ii).
\end{proof}

\indent When Hilbert space are considered, Theorem \ref{thm5.3} reduces to the following corollary.\par

\begin{corollary}\label{cor5.10} Let $H$ be a Hilbert space and consider $T$, $E$ and $F\in L(H)$ such that
$E$ and $F$ are invertible positive. Then, the following statement are equivalent.\par
\noindent \rm (i) \it $T$ is weighted EP with weights $E$ and $F$.\par
\noindent \rm (ii) \it There exists an idempotent $P\in L(H)$ such that $E^{-1}P^*E=P$,
 $F^{-1}P^*F=P$, and if $H_1=R(P)$, $H_2=N(P)$ and $T_1=T\mid_{H_1}^{H_1}\colon H_1\to H_1$, then $H=H_1\oplus H_2$ and
with respect to the aforementioned decomposition, $T$ and $T^\dag_{E,F}$ can be presented as
$$
T= \begin{pmatrix} T_1&0\\ 0&0\\\end{pmatrix}, \hskip2truecm
T^\dag_{E,F}= \begin{pmatrix} T_1^{-1}&0\\ 0&0\\\end{pmatrix},
$$
respectively.
\end{corollary}
\begin{proof} Apply Theorem \ref{thm5.3} to the case under consideration and recall that according to
\cite[Proposition 20, Chapter I, section 12]{BD}, $P\in H(L(H)^E)$ (respectively $P\in H(L(H)^F)$)
if and only if $E^{-1}P^*E=P$ (respectively $F^{-1}P^*F=P$), see the discussion before \cite[Definition 2.3]{BMDj}.
\end{proof}
\indent Next the Banach algebra case will be studied. Recall that according to \cite[Theorem 1.4]{DKS},
if $A$ is a $C^*$-algebra, then $a\in A$ is EP if and only if there exists a hermitian idempotent $q\in A$ such that
$qa=a=aq$ and $a=qaq\in (qAq)^{-1}$, where given a Banach algebra $B$ and $p\in B$ such that $p^2=p$, then $pBp$ is a Banach algebra with unit $p$.\par

\begin{theorem}\label{thm5.5} Let $A$ be a unital Banach algebra and consider $e$, $f\in A$
two invertible positive elements. Let $a\in A$. Then the following statements are equivalent.\par
\noindent {\rm (i)}The element $a$ is weighted EP with weights $e$ and $f$.\par
\noindent {\rm (ii)} There exist $c\in A^{-1}$ and $p\in A$ such that $p=p^2$, $cpc^{-1}\in H(A^e)\cap H(A^f)$, $pap\in (pAp)^{-1}$ and $a=cpapc^{-1}$. \par
\noindent {\rm (iii)}The element $a$ is weighted EP with weights $f$ and $e$.\par
\noindent {\rm (iv)}The element $a$ is weighted EP both with weights $e$ and $e$ and with weights $f$ and $f$.
\end{theorem}
\begin{proof} If $a\in A$ is weighted EP with weights $e$ and $f$, then $p=aa^{\dag}_{e,f}=a^{\dag}_{e,f}a$ and $c=1=p+(1-p)$
satisfy the conditions in statement (ii). Note that in this case $a=pap$ and $pa^{\dag}_{e,f}p=a^{\dag}_{e,f}\in pAp$ is the inverse of
$a=pap$ in the subalgebra $pAp$.\par

\indent On the other hand, suppose that statement (ii) holds. Since $pap\in (pAp)^{-1}$, there exists $b\in A$ such that $papbp=pbpap=p$.
As a result, $d=cpbpc^{-1}$ is a normalized generalized inverse of $a$
such that $ad=da=cpc^{-1}$. In fact,
\begin{align*}
ad&= acpbpc^{-1}= cpapc^{-1}cpbpc^{-1}= cpc^{-1},\\
da&= cpbpc^{-1}cpapc^{-1}=cpc^{-1},\\
ada&= acpc^{-1}= cpapc^{-1}cpc^{-1}=cpapc^{-1}=a\\
 dad&= cpc^{-1}cpbpc^{-1}=cpbpc^{-1}=d.\\
\end{align*}
Since $cpc^{-1}\in H(A^e)\cap H(A^f)$, $a^{\dag}_{e,f}$ exists and
$a$ is weighted EP with weights $e$ and $f$. In fact, $a^{\dag}_{e,f}=d$.\par
\indent To prove the equivalences among statement (i) and statements (iii) and (iv), use an
argument similar to the one in the proof of Corollary \ref{cor5.9}.
\end{proof}
\indent Under the same hypothesis of Theorem \ref{thm5.5}, when $e=f=1$, the following corollary
presents the case $a\in A$, $a$ an EP element, see \cite[section 5]{B3}.\par

\begin{theorem}\label{thm5.11} Let $A$ be a unital Banach algebra and consider $a\in A$. Then the following statements are equivalent.\par
\noindent {\rm (i)}The element $a$ is EP.\par
\noindent {\rm (ii)} There exist $c\in A^{-1}$ and $p\in A$ such that $p=p^2$, $cpc^{-1}\in H(A)$, $pap\in (pAp)^{-1}$ and $a=cpapc^{-1}$.
\end{theorem}
\begin{proof} Apply Theorem \ref{thm5.5} to the case under consideration.
\end{proof}

\indent As a corollary of Theorem \ref{thm5.5}, weighted EP $C^*$-algebra elements will be characterized.\par

\begin{corollary}\label{cor14} Let $A$ be a $C^*$-algebra and consider $e$, $f\in A$
two invertible positive elements. Let $a\in A$. Then the following statements are equivalent.\par
\noindent {\rm (i)}The element $a$ is weighted EP with weights $e$ and $f$.\par
\noindent {\rm (ii)} There exist $c\in A^{-1}$ and $p\in A$ such that $p=p^2$, $e(cpc^{-1})^*e^{-1}=cpc^{-1}$,
$f(cpc^{-1})^*f^{-1}=cpc^{-1}$ and $a=cpapc^{-1}$. \par
\end{corollary}
\begin{proof} Adapt the proof of Corollary \ref{cor5.10} to the case under consideration applying
Theorem \ref{thm5.5} instead of Theorem \ref{thm5.3}.
\end{proof}

\bibliographystyle{amsplain}

\vskip.3truecm
\noindent Enrico Boasso\par
\noindent  E-mail address: enrico\_odisseo@yahoo.it\par
\vskip.3truecm
\noindent Dragan S. Djordjevi\'c\par
\noindent  E-mail address: dragan@pmf.ni.ac.rs\par
\vskip.3truecm
\noindent Dijana Mosi\'c\par
\noindent  E-mail address: dijana@pmf.ni.ac.rs\par
\end{document}